\documentclass[journal,web]{ieeecolor}
\usepackage{multirow}
\usepackage{float}
\usepackage{generic}
\usepackage{cite}
\usepackage{amsmath,amssymb,amsfonts}
\usepackage{graphicx}
\usepackage{algorithm}
\usepackage{soul} 
\usepackage{xcolor}
\usepackage{tikz} 
\usepackage{color}
\usepackage{graphicx}
\graphicspath{{figures/}{./}}
\usepackage{textcomp}
\usepackage{trimclip}
\usepackage[mathscr]{euscript}
\usepackage{array}
\usepackage{eqparbox}
\usepackage{url}
\usepackage{relsize}
\usepackage{dsfont}
\usepackage{nopageno}

\usepackage{epsfig}
\usepackage{amssymb}
\usepackage[mathscr]{euscript}

\usepackage{tikz-cd}
\usepackage{algorithm}
\usepackage{algpseudocode}

\newtheorem{theorem}{Theorem}

\newtheorem{proposition}{Proposition}

\newtheorem{lemma}{Lemma}
\usepackage{lipsum} 
\usepackage{booktabs}

\def\BibTeX{{\rm B\kern-.05em{\sc i\kern-.025em b}\kern-.08em
T\kern-.1667em\lower.7ex\hbox{E}\kern-.125emX}}

\begin{document}

\title{Model Reduction of Homogeneous Polynomial Dynamical Systems via Tensor Decomposition}

\author{Xin Mao, Can Chen
\thanks{Xin Mao is with the School of Data Science and Society, University of North Carolina at Chapel Hill, Chapel Hill, NC 27599, USA (email: xinm@unc.edu).}
\thanks{Can Chen is with the School of Data Science and Society, the Department of Mathematics, and the Department of Biostatistics,  University of North Carolina at Chapel Hill, Chapel Hill, NC 27599, USA (email: canc@unc.edu).}%
}

\maketitle

\thispagestyle{empty}
\pagestyle{empty}

%%%%%%%%%%%%%%%%%%%%%%%%%%%%%%%%%%%%%%%%%%%%%%%%%%%%%%%%%%%%%%%%%%%%%%%%%%%%%%%%
\begin{abstract}
Model reduction plays a critical role in system control, with established methods such as balanced truncation widely used for linear systems. However, extending these methods to nonlinear settings, particularly polynomial dynamical systems that are often used to model higher-order interactions in physics, biology, and ecology, remains a significant challenge. In this article, we develop a novel model reduction method for homogeneous polynomial dynamical systems (HPDSs) with linear input and output grounded in tensor decomposition. Leveraging the inherent tensor structure of HPDSs, we construct reduced models by extracting dominant mode subspaces via higher-order singular value decomposition. Notably, we establish that key system-theoretic properties, including stability, controllability, and observability, are preserved in the reduced model. We demonstrate the effectiveness of our method using numerical examples.
\end{abstract}

%%%%%%%%%%%%%%%%%%%%%%%%%%%%%%%%%%%%%%%%%%%%%%

\section{Introduction}
Many real-world systems, including social networks \cite{cencetti2021temporal,alvarez2021evolutionary,BOCCALETTI20231}, ecological networks \cite{grilli2017higher,chen2024stability, bairey2016high}, and chemical reaction networks \cite{chen2022explicit,mann2023ai,10535389}, exhibit higher-order interactions that can be effectively captured using polynomial dynamical systems or homogeneous polynomial dynamical systems (HPDSs). Such systems are typically high-dimensional, posing significant computational challenges for analysis, simulation, and control. Model reduction offers a powerful strategy to address these challenges by approximating the original high-dimensional system with a lower-dimensional surrogate that retains essential dynamical features and system-theoretic properties \cite{penzl2006algorithms,benner2015survey,bassi2003dynamical,baur2014model,touze2021model,kawano2016model,cheng2018model}. It has demonstrated considerable success across a wide range of applications, including fluid dynamics \cite{lassila2014model,carlberg2013gnat}, structural health monitoring \cite{taddei2018simulation,sengupta2025state}, and coarse-graining in multiscale modeling \cite{gorban2006model,budarapu2014efficient}. However, traditional model reduction methods such as balanced truncation are primarily tailored to linear systems and do not readily extend to the nonlinear, polynomial context.

Several efforts have been devoted to developing model reduction methods for nonlinear systems. Notable examples include empirical balanced truncation \cite{condon2004empirical}, nonlinear balanced truncation \cite{sahyoun2013reduced}, moment matching \cite{astolfi2010model}, proper orthogonal decomposition \cite{pinnau2008model}, and techniques based on Koopman operator theory \cite{peitz2019koopman,santos2021reduced}. While these methods have shown promise in specific settings, many face limitations when applied to systems with strong nonlinearities, such as those with high polynomial degrees. Prajna and Sandberg \cite{prajna2005model} developed a computational method for model reduction of polynomial dynamical systems by leveraging sum-of-squares relaxations applied to specific Lyapunov inequalities, which act as nonlinear analogues to the Lyapunov controllability and observability linear matrix inequalities used in linear systems theory. However, this approach requires computing polynomial controllability and observability Gramians through sum-of-squares programming, rendering it computationally intractable for large-scale systems. This highlights a pressing need for scalable and efficient model reduction methods tailored to polynomial dynamical systems.

Recently, tensor-based methods have garnered increasing attention for applications to HPDSs, as every HPDS admits an equivalent representation in tensor form \cite{chen2022explicit,chen2024stability2}. This tensor representation enables the efficient analysis of key system properties, such as stability, controllability, and observability, in a compact and mathematically tractable framework \cite{mao2025tensor, chen2021controllability,10540364,10910206,pickard2023observability}. As a result, tensor decomposition emerges as a promising and effective tool for performing model reduction in HPDSs, a direction that, to the best of our knowledge, has not yet been explored. Among the various tensor decomposition techniques available, we are interested in higher-order singular value decomposition (HOSVD), a natural multilinear generalization of the matrix SVD \cite{de2000multilinear,5447070}. A significant benefit of HOSVD lies in its generation of orthogonal factor matrices along each tensor mode, enabling structured dimensionality reduction while preserving the energy distribution and geometric structure of the original system. Additionally,  it facilitates coordinate transformations that simplify the analysis of stability, controllability, and observability in the reduced model.

In this article, we develop a novel model reduction method for HPDSs with linear input and output via tensor decomposition. Based on the fact that any polynomial dynamical system can be transformed into an HPDS through homogenization, the proposed method is readily applicable to general polynomial dynamical systems. The key contributions are as follows:
\begin{itemize}
    \item We perform model reduction  for tensor-based HPDSs with linear input and output by leveraging the compact HOSVD of the system's dynamic tensor, obtaining a reduced model that preserves dominant tensor modes. 
    \item We establish that key system-theoretic properties, including stability, controllability, and observability, are preserved in the reduced model under certain structural assumptions of the dynamic tensor.
    \item We verify the effectiveness of the proposed  model reduction method through numerical simulations. 
\end{itemize}
Model reduction for HPDSs or  polynomial dynamical systems has broad potential applications in domains such as network dynamics, where complex higher-order interactions and high-dimensional state spaces are prevalent. For example, in ecological modeling, reduced polynomial models can support the coarse-graining of large-scale  network dynamics, enabling scalable analysis, efficient simulation, and informed decision-making for ecological conservation \cite{https://doi.org/10.1111/ele.12893,Tsoumanis_2012}.

This article is organized into six sections. Section~\ref{sec:pre} reviews essential concepts in tensor algebra, including tensor unfolding, tensor matrix/vector multiplication, tensor eigenvalues, and tensor decomposition. Section~\ref{sec:mr} details the proposed model reduction procedure for HPDSs with linear input and output, leveraging HOSVD. Section~\ref{sec:pro} establishes theoretical results showing that key system-theoretic properties are preserved in the reduced model. Section~\ref{sec:num} presents numerical examples to demonstrate the effectiveness of the proposed method. Finally, Section~\ref{sec:con} presents concluding remarks and highlights future research directions.

\section{Tensor Preliminaries}\label{sec:pre}
Tensors are multidimensional arrays that generalize vectors and matrices \cite{kolda2009tensor,7891546,brand2020vector,chen2024tensor}. The number of dimensions of a tensor is called  order, and each dimension is referred to as a mode. A $k$th-order tensor is often denoted by $\mathscr{A} \in \mathbb{R}^{n_1 \times n_2 \times \cdots \times n_k}$. A tensor is said to be cubical if all its modes have the same dimension, i.e., $n_1=n_2=\cdots=n_k$. Moreover, a cubical tensor $\mathscr{A}$ is called symmetric if its entries $\mathscr{A}_{j_1 j_2 \cdots j_k}$ remain invariant under any permutation of the indices, and it is termed almost symmetric if the entries are invariant under any permutation of the first $k - 1$ indices.

\subsection{Tensor Unfolding}
Tensor unfolding, also known as matricization or flattening, refers to the process of rearranging the entries of a tensor into a matrix \cite{ragnarsson2012block}. For a $k$th-order tensor $\mathscr{A} \in \mathbb{R}^{n_1 \times n_2 \times \cdots \times n_k}$, its $p$-mode unfolding, denoted by $\textbf{A}_{(p)}\in\mathbb{R}^{n_p\times(n_1n_2\cdots n_{p-1}n_{p+1}\cdots n_k)}$,  is obtained by mapping the $p$-mode fiber  of $\mathscr{A}$ (obtained by fixing all indices except the $p$th) to the columns of the resulting matrix while preserving the ordering of the other modes. Formally, the entries of $\mathscr{A}$ are defined as
\begin{equation*}
    (\textbf{A}_{(p)})_{j_pj} = \mathscr{A}_{j_1 j_2 \ldots j_k} \text{ where } j = j_1+\sum_{\substack{i=2 \\ i\neq p}}^k(j_i-1)\underset{\substack{l=1 \\ l\neq p}}{\overset{i-1}{\prod}}n_l.
\end{equation*}
This operation is fundamental in tensor computations, especially in tensor decomposition.

\subsection{Tensor Matrix/Vector Multiplication}
Tensor matrix multiplication is a generalization of matrix matrix multiplication to higher-order tensors \cite{kolda2009tensor}. Let \(\mathscr{A} \in \mathbb{R}^{n_1 \times n_2 \times \cdots \times n_k}\) be a \(k\)th-order tensor and \(\textbf{M} \in \mathbb{R}^{m \times n_p}\) be a matrix. The $p$-mode multiplication of \(\mathscr{A}\) with \(\textbf{M}\) along mode $p$, denoted by \(\mathscr{A} \times_p \textbf{M}\in\mathbb{R}^{n_1 \times  \cdots \times n_{p-1} \times m \times n_{p+1} \times \cdots \times n_k}\), is defined as
\[
(\mathscr{A} \times_p \textbf{M})_{j_1 \cdots j_{p-1} i j_{p+1} \cdots j_k} = \sum_{j_p=1}^{n_p} \mathscr{A}_{j_1 j_2 \cdots j_k} \textbf{M}_{i j_p}.
\]
When $m=1$, this operation  reduces to  $p$-mode tensor vector multiplication. Specifically, for a vector  $\textbf{v}\in\mathbb{R}^{n_p}$, $\mathscr{A} \times_p \textbf{v}\in\mathbb{R}^{n_1\times \cdots\times n_{p-1}\times n_{p+1}\times \cdots\times n_k}$ is defined as
\[
(\mathscr{A} \times_p \textbf{v})_{j_1 \cdots j_{p-1}  j_{p+1} \cdots j_k} = \sum_{j_p=1}^{n_p} \mathscr{A}_{j_1 j_2 \cdots j_k} \textbf{v}_{j_p}.
\]
Tensor matrix/vector multiplication can be applied sequentially along all modes of $\mathscr{A}$, which is known as the Tucker product.

\subsection{Tensor Eigenvalue}
The concept of eigenvalues for tensors generalizes the matrix eigenvalue problem \cite{qi2005eigenvalues,qi2007eigenvalues}. Several definitions exist, but one of the most widely used in the context of symmetric tensors is the Z-eigenvalue \cite{qi2005eigenvalues}. Let $\mathscr{A}\in \mathbb{R}^{n \times n \times \stackrel{k}{\cdots} \times n}$ be a $k$th-order symmetric tensor.  A real number $\lambda \in \mathbb{R}$ is called a Z-eigenvalue of $\mathscr{A}$ if there exists a nonzero vector $\textbf{x} \in \mathbb{R}^n$ (i.e., the corresponding Z-eigenvector) such that
\[
\mathscr{A}\times_1 \textbf{x}\times_2 \textbf{x}\times_3\cdots\times_{k-1} \textbf{x} = \lambda \textbf{x}, \quad \text{subject to } \|\textbf{x}\|_2 = 1.
\]
Every real symmetric tensor admits at least one Z-eigenvalue. Moreover, Z-eigenvalues are invariant under orthogonal similarity transformations. If $\textbf U \in \mathbb{R}^{n \times n}$ is an orthogonal matrix, $\mathscr{A}$ and $\mathscr{A} \times_1 \textbf U \times_2 \textbf U\times_3 \cdots \times_k \textbf U$ share the same Z-eigenvalues. The computation of Z-eigenvalues involves solving a constrained polynomial optimization problem, which is NP-hard in general. Nevertheless, many algorithms, such as power methods, polynomial optimization, and homotopy continuation, have been developed to compute or approximate them effectively.

\subsection{Tensor Decomposition}
Tensor decomposition refers to the process of expressing a tensor as a combination of simpler, lower-dimensional components that reveal its underlying multilinear structure \cite{kolda2009tensor,oseledets2011tensor,de2000multilinear}. Among the various decomposition techniques, the higher-order singular value decomposition (HOSVD) \cite{de2000multilinear} is of particular interest  due to its ability to produce orthogonal factor matrices along each mode, thereby facilitating structured dimensionality reduction and preserving key properties of the original tensor. Given a $k$th-order tensor $\mathscr{A} \in \mathbb{R}^{n_1 \times n_2 \times \cdots \times n_k}$, the HOSVD of $\mathscr{A}$ is expressed as
\begin{equation}\label{eq:hosvd}
    \mathscr{A} = \mathscr{S} \times_1 \textbf U_{1} \times_2 \textbf U_2 \times_3\cdots \times_k \textbf U_k,
\end{equation}
where $\mathscr{S}\in\mathbb R^{n_1\times n_2\times\dots\times n_k}$ is called the core tensor, and  $\textbf U_p\in\mathbb R^{n_p\times n_p}$ are  orthogonal matrices (see Fig. \ref{fig:hosvd}). Let $\mathscr{S}_{j_p=\alpha}$ denote the sub-tensor obtained by fixing the $p$th index to $\alpha$. The Frobenius norms $\|\mathscr{S}_{j_p=\alpha}\|$ are termed the $p$-mode higher-order singular values of $\mathscr{A}$, which satisfy the ordering 
\[\|\mathscr{S}_{j_p=1}\|\geq \|\mathscr{S}_{j_p=2}\|\geq \cdots\geq \|\mathscr{S}_{j_p=n_p}\|\geq 0
\]
for all $p$. When many of these higher-order singular values are zero or negligibly small, the HOSVD \eqref{eq:hosvd} admits a compact representation (analogous to compact or economy-size matrix SVD), i.e., 
\[
\mathscr{A} = \mathscr{S}_{\text{red}} \times_1 \textbf V_{1} \times_2 \textbf V_2 \times_3\cdots \times_k \textbf V_k,
\]
where $\mathscr{S}_{\text{red}}\in\mathbb{R}^{r_1\times r_2\times \cdots\times r_k}$ is the reduced core tensor, and $\textbf{V}_p\in\mathbb{R}^{n_p\times r_p}$ are factor matrices containing the first $r_p$ columns of $\textbf{U}_p$. The full/compact HOSVD can be computed by performing  matrix SVDs on the $p$-mode unfoldings of $\mathscr{A}$, where $\textbf{U}_p$ or $\textbf{V}_p$ are the left singular vector matrices of $\textbf{A}_{(p)}$.

\begin{figure}[t]
\begin{center}		
\includegraphics[width=\linewidth]{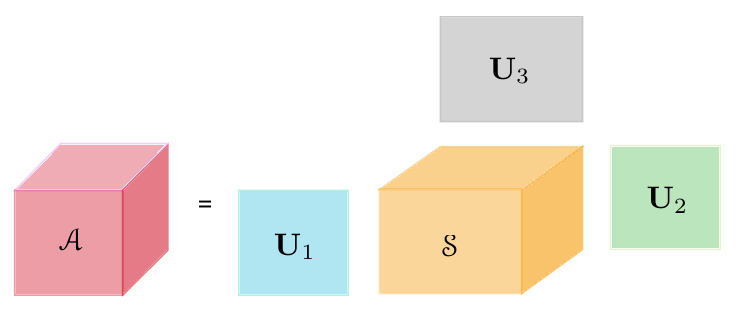}    % The printed column width is 8.4 cm.
\caption{An example of the HOSVD of a third-order tensor.}\label{fig:hosvd}
\end{center}
\end{figure}

Tensor orthogonal decomposition is a special case of HOSVD that applies specifically to symmetric tensors \cite{robeva2016orthogonal}. In this setting, the orthogonal factor matrices along each mode are identical (i.e., $\textbf{U}_1=\textbf{U}_2=\cdots=\textbf{U}_k=\textbf{U}$), and the resulting core tensor $\mathscr{S}$ becomes diagonal (i.e., $\mathscr{S}_{j_1j_2\cdots j_k}=0$ except $j_1=j_2=\cdots=j_k$). Given a $k$th-order symmetric tensor $\mathscr{A}\in \mathbb{R}^{n \times n \times \stackrel{k}{\cdots} \times n}$, the orthogonal decomposition of $\mathscr{A}$ can be equivalently represented as
\[
\mathscr{A} = \sum_{j=1}^n \lambda_j \textbf u_{j}\circ \textbf u_{j} \circ \stackrel{k}{\cdots} \circ\textbf u_{j},
\]
where $\circ$ denotes the outer product, $\lambda_j$ are the diagonal entries of $\mathscr{S}$, and $\textbf{u}_j\in\mathbb{R}^{n}$ are the $j$th columns of the common orthogonal factor matrix $\textbf{U}$. It has been shown that $\lambda_j$ are the Z-eigenvalues of $\mathscr{A}$ with the corresponding Z-eigenvectors $\textbf{u}_j$. A symmetric tensor that admits such a decomposition is referred to as orthogonally decomposable (odeco).

\section{Model Reduction of Input-Output HPDSs}\label{sec:mr}
Every homogeneous polynomial dynamical system (HPDS) of degree $k-1$ can be uniquely represented by a $k$th-order almost symmetric tensor via tensor vector multiplications along the first $k-1$ modes \cite{chen2022explicit,chen2024stability2}. Consider a continuous-time tensor-based HPDS with linear input and output 
\begin{align}\label{eq:syso}
\begin{cases}
\dot{\textbf x}(t)=\mathscr A\textbf x(t)^{k-1}+\textbf B\textbf u(t) \\
\textbf y(t)=\textbf C\textbf x(t)
\end{cases},
\end{align}
where $\mathscr A\in\mathbb{R}^{n\times n\times\stackrel{k}{\cdots}\times n}$ is an almost symmetric dynamic tensor, $\textbf B\in\mathbb{R}^{n\times m}$ is the control matrix, $\textbf C\in\mathbb{R}^{l\times n}$ is the output matrix. Here, $\textbf x(t)\in\mathbb R^n$ is the state, $\textbf u(t)\in\mathbb{R}^m$ is the control input, and $\textbf y(t)\in\mathbb{R}^l$ is the  output. For simplicity, we denote $\mathscr A\textbf x(t)^{k-1} = \mathscr{A}\times_1\textbf{x}(t)\times_2\textbf{x}(t)\times_3\cdots\times_{k-1}\textbf{x}(t)$.

 \begin{algorithm}[t]
\caption{HOSVD-based model reduction of tensor-based HPDSs with linear input and output}\label{alg:MOR}
\begin{algorithmic}[1]
\State Given an input-output tensor-based HPDS with dynamic tensor $\mathscr A\in\mathbb{R}^{n\times n\times\stackrel{k}{\cdots}\times n}$, input matrix $\textbf B\in\mathbb{R}^{n\times m}$, and output matrix $\textbf C\in\mathbb{R}^{l\times n}$
\State Compute the compact HOSVD of $\mathscr A$ with respect to the first $k-1$ modes, i.e.,
$\mathscr{A} = \mathscr{S} \times_1 \textbf V \times_2 \textbf V\times_3 \cdots \times_{k-1}\textbf V\times_k \textbf V_k$
\State Set the new latent state $\textbf z(t)=\textbf V^\top\textbf x(t)\in\mathbb R^r$
\State The reduced input-output HPDS with system parameters computed as
\begin{align*}
\mathscr{A}_{\text{red}}&=\mathscr{S}_{\text{red}}\times_k \textbf{V}^\top\textbf{V}_k\in\mathbb{R}^{r\times r\times \stackrel{k}{\cdots} \times r}\\
    \textbf{B}_{\text{red}}&=\textbf{V}^\top\textbf{B}\in\mathbb{R}^{r\times m}, \quad \textbf{C}_{\text{red}}= \textbf{C}\textbf{V}\in\mathbb{R}^{l\times r}
\end{align*}

\State\Return The reduced HPDS with linear input and output.
\end{algorithmic}
\end{algorithm}

Since the dynamic tensor $\mathscr{A}$ is almost symmetric, the compact HOSVD  of $\mathscr{A}$  is computed as
\[
\mathscr{A}= \mathscr{S}_{\text{red}} \times_1 \textbf V \times_2 \textbf V \times_3\cdots \times_{k-1}\textbf{V}\times_k \textbf V_k,
\]
where $\mathscr{S}_{\text{red}}\in\mathbb{R}^{(r\times r\times \stackrel{k-1}{\cdots} \times r)\times r_k}$, $\textbf{V}\in\mathbb{R}^{n\times r}$, and $\textbf{V}_k\in\mathbb{R}^{n\times r_k}$.  Let $\textbf{z}(t)=\textbf{V}^{\top}\mathbf{x}(t)\in\mathbb{R}^{r}$ denote the new latent state and substitute the dynamic tensor $\mathscr{A}$ with its compact HOSVD form. According to the properties of tensor vector multiplication, the new state dynamics is given by
\begin{align*}
\dot{\textbf{z}}(t) &= \textbf{V}^\top\dot{\textbf{x}}(t) = \textbf{V}^\top \big(\mathscr A\textbf x(t)^{k-1} + \textbf{B} \textbf u(t)\big)\\
& = (\mathscr{S}_{\text{red}}\times_k \textbf{V}^\top\textbf{V}_k)(\textbf{V}^\top\textbf{x}(t))^{k-1} + \textbf{V}^\top\textbf{B}\textbf{u}(t)\\
& = (\mathscr{S}_{\text{red}}\times_k \textbf{V}^\top\textbf{V}_k)\textbf{z}(t)^{k-1}+ \textbf{V}^\top\textbf{B}\textbf{u}(t),
\end{align*}
and the output $\textbf{y}(t) = \textbf{C}\textbf{V}\textbf{z}(t)$. Define $\mathscr{A}_{\text{red}}=\mathscr{S}_{\text{red}}\times_k \textbf{V}^\top\textbf{V}_k\in\mathbb{R}^{r\times r\times \stackrel{k}{\cdots} \times r}$, $\textbf{B}_{\text{red}}=\textbf{V}^\top\textbf{B}\in\mathbb{R}^{r\times m}$, and $\textbf{C}_{\text{red}}= \textbf{C}\textbf{V}\in\mathbb{R}^{l\times r}$. Then the reduced HPDS with linear input and output can be expressed as
\begin{align}\label{eq:redsyso}
\begin{cases}
\dot{\textbf z}(t)=\mathscr{A}_{\text{red}}\textbf z(t)^{k-1}+\textbf{B}_{\text{red}}\textbf u(t) \\
\textbf y(t)=\textbf{C}_{\text{red}}\textbf z(t)
\end{cases}.
\end{align}

The detailed steps of the model reduction procedure are presented in Algorithm~\ref{alg:MOR}. The proposed model reduction method leverages the low-rank structure inherent in the dynamic tensor through HOSVD, significantly reducing the number of system parameters from $n^k+nm+nl$ to $r^{k}+rm+rl$ while preserving the essential nonlinear dynamics of the original input-output HPDS \eqref{eq:syso}. In the following, we demonstrate that system-theoretic properties, including stability, controllability, and observability, are preserved in the reduced model.

\section{System-Theoretic Properties}\label{sec:pro}
We aim to rigorously examine whether fundamental system-theoretic properties, such as stability, controllability, and observability, are retained in the reduced  model \eqref{eq:redsyso}. These properties are essential for ensuring that the reduced model not only approximates the behavior of the original system but also remains suitable for analysis and control design.

\subsection{Stability}
Determining the stability of nonlinear HPDSs is generally challenging \cite{8706528}. Recent studies have shown that when the dynamic tensor $\mathscr{A}$ is symmetric and odeco, the stability of a tensor-based HPDS can be characterized by the Z-eigenvalues of the dynamic tensor \cite{chen2022explicit}. Therefore, we focus our stability analysis on tensor-based HPDSs with odeco dynamic tensors (referred to as odeco HPDSs). 

Consider an autonomous odeco HPDS 
\begin{align}\label{eq:sysso}
\dot{\textbf x}(t)=\mathscr A\textbf x(t)^{k-1}
\end{align}
with the reduced counterpart  obtained through the proposed HOSVD-based model reduction method
\begin{align}\label{eq:redsysso}
\dot{\textbf z}(t)=\mathscr A_{\text{red}}\textbf z(t)^{k-1}.
\end{align}
While the odeco HPDS \eqref{eq:sysso} may admit infinitely many equilibrium points, their dynamic behaviors are the same as the one at the origin. The equilibrium point $\textbf x_{\text{e}}=\textbf{0}$ is said to be stable if $\|\textbf{x}(t)\|\leq \gamma\|\textbf{x}\|_0$ for initial condition $\textbf{x}_0$ and $\gamma>0$, asymptotically stable if $\|\textbf{x}(t)\|\rightarrow 0$ as $t\rightarrow 0$, and unstable if $\|\textbf{x}(t)\|\rightarrow \infty$ as $t\rightarrow c$ for $c>0$. The stability properties of the odeco HPDS \eqref{eq:sysso} are summarized below. 

\begin{theorem}[Stability \cite{chen2022explicit}]\label{thm:stability}
    Suppose that the initial condition for the odeco HPDS \eqref{eq:sysso} is $\textbf{x}_0=\sum_{j=1}^n\alpha_j \textbf{u}_j$ where $\textbf{u}_j$ are the Z-eigenvectors of $\mathscr{A}$. The equilibrium point $\textbf x_{\text{e}}=\textbf{0}$ of the odeco HPDS \eqref{eq:sysso} is (i) stable if and only if $\lambda_j\alpha_j^{k-2}\leq 0$ for all $j=1,2,\dots, n$; (ii) asymptotically stable if and only if $\lambda_j\alpha_j^{k-2}< 0$ for all $j=1,2,\dots, n$; (iii) unstable if and only if $\lambda_j\alpha_j^{k-2}> 0$ for some $j=1,2,\dots, n$,
    where $\lambda_j$ are the corresponding Z-eigenvalues of $\textbf{u}_j$. 
\end{theorem}

We first show that the reduced HPDS \eqref{eq:redsysso} preserves the odeco structure.

\begin{lemma}\label{lemma:1}
The reduced dynamic tensor $\mathscr{A}_{\text{red}}$ is odeco and retains the nonzero Z-eigenvalues of $\mathscr{A}$. Moreover, the reduced HPDS \eqref{eq:redsysso} has a unique equilibrium point at the origin. 
\end{lemma}
\begin{proof}
  Since the dynamic tensor $\mathscr{A}$ is odeco, the core tensor $\mathscr{S}$ in its HOSVD is diagonal, with the Z-eigenvalues of $\mathscr{A}$ appearing along the diagonal entries. Therefore, the reduced core tensor $\mathscr{S}_{\text{red}}$ is a diagonal tensor containing nonzero Z-eigenvalues. Furthermore, because $\mathscr{A}$ is symmetric, the last factor matrix in its HOSVD satisfies $\textbf{V}_k = \textbf{V}$, which implies 
  \begin{align*}
\mathscr S_\text{red}&=\Big(\sum_{j=1}^n\lambda_j\textbf u_j\circ\textbf u_j\circ\dots\circ\textbf u_j\Big)\times_1\textbf V^\top\times_2\dots\times_k\textbf V^\top\\&=\sum_{j=1}^r\lambda_j(\textbf V^\top\textbf u_j\circ\textbf V^\top\textbf u_j\circ\dots\circ\textbf V^\top\textbf u_j)\\&=\sum_{j=1}^r\lambda_j\textbf e_j\circ \textbf e_j\circ\dots\circ\textbf e_j,
\end{align*}
where $\textbf{e}_j$ are the standard basis vectors. Therefore, $\mathscr{A}_{\text{red}} = \mathscr{S}_{\text{red}}$ is odeco, and $\mathscr{A}_{\text{red}}$ retains nonzero Z-eigenvalues in its orthogonal decomposition. Moreover, by Proposition 2 in \cite{chen2022explicit}, the reduced HPDS \eqref{eq:redsysso} then has a unique equilibrium point at the origin, as all Z-eigenvalues of $\mathscr{A}_{\text{red}}$ are nonzero.
\end{proof}

Building on the above lemma, we can  show the stability properties of the reduced HPDS \eqref{eq:redsysso}.

\begin{proposition}\label{pro:stability}
If the original odeco HPDS \eqref{eq:sysso} is stable or asymptotically stable at an equilibrium point, then the reduced HPDS \eqref{eq:redsysso} is asymptotically stable at the origin. Moreover, if the original system is unstable at an equilibrium point, the reduced system is also unstable at the origin. 
\end{proposition}
\begin{proof}
Suppose that the initial condition of the odeco HPDS \eqref{eq:sysso} is given by $\textbf{x}_0=\sum_{j=1}^n\alpha_j \textbf{u}_j$ and the transformation $\textbf{z}(t) = \textbf{V}^\top\textbf{x}(t)$. The initial condition for the reduced HPDS \eqref{eq:redsysso} can be computed as
    \[
    \textbf{z}_0 = \textbf{V}^\top\textbf{x}_0=\textbf{V}^\top\textbf{U}\begin{bmatrix}
        \alpha_1\\
        \alpha_2\\
        \vdots\\
        \alpha_n
    \end{bmatrix} = \begin{bmatrix}
         \alpha_1\\
        \alpha_2\\
        \vdots\\
        \alpha_r
    \end{bmatrix}  = \sum_{j=1}^r \alpha_j\textbf{e}_j,
    \]
where $\textbf{U}\in\mathbb{R}^{n\times n}$  contains the Z-eigenvectors $\textbf{u}_j$ as its columns. If the original odeco HPDS \eqref{eq:sysso} is stable or asymptotically stable at an equilibrium point, we have $\lambda_j\alpha_j^{k-2}\leq 0$ for $j=1,2,\dots,n$ by Theorem~\ref{thm:stability}. Applying Lemma \ref{lemma:1} yields that $\lambda_j\alpha_j^{k-2}< 0$ for $j=1,2,\dots,r$ and $\textbf{x}_{\text{e}}$ is the unique equilibrium point. Therefore, we conclude that the reduced HPDS \eqref{eq:redsysso} is asymptotically stable at the origin. The unstable case follows similarly.
\end{proof}

Focusing on the class of odeco HPDSs, whose stability can be characterized by the Z-eigenvalues of their dynamic tensors, we demonstrate that stability properties are preserved under the proposed HOSVD-based model reduction method.

\subsection{Controllability}
The controllability of input HPDSs has been widely studied, particularly yielding strong results in the case of linear control input \cite{jurdjevic1985polynomial}. Consider a general tensor-based HPDS with linear input (i.e., the dynamic tensor $\mathscr{A}$ is almost symmetric) 
\begin{align}\label{eq:sysco}
\dot{\textbf x}(t)=\mathscr A\textbf x(t)^{k-1}+\textbf B\textbf u(t)
\end{align}
with the reduced counterpart derived from the proposed
HOSVD-based model reduction method
\begin{align}\label{eq:redsysco}
\dot{\textbf z}(t)=\mathscr A_{\text{red}}\textbf z(t)^{k-1}+\textbf B_{\text{red}}\textbf u(t).
\end{align}

A dynamical system is said to be strongly controllable if, for any pair of initial and target states, there exists a suitable control input that drives the system from the initial state to the target state within an arbitrarily chosen time interval \cite{jurdjevic1985polynomial}. For the tensor-based HPDS of odd degree (i.e., $k$ is even) with linear input \eqref{eq:sysco}, it has been shown that strong controllability can be verified using a tensor-based rank criterion. This criterion extends the classical Kalman's rank condition from linear systems to the setting of HPDSs \cite{mao2025tensor,chen2021controllability}.

\begin{theorem}[Controllability \cite{mao2025tensor}]\label{thm:control}
    The tensor-based HPDS of with linear input \eqref{eq:sysco} for even $k$ is strongly controllable if and only if the controllability matrix defined as
    \begin{equation}\label{eq:controllability}
        \textbf R=\begin{bmatrix}\textbf M_0 & \textbf M_1 &\cdots& \textbf M_{n-1}\end{bmatrix},
    \end{equation}
where $\textbf M_0=\textbf B$ and 
each $\textbf M_j$ is formed from 
\begin{align*}
    \{&\mathscr A\times_1\textbf  v_1\times_2\textbf{v}_2\times_3\cdots\times_{k-1}\textbf v_{k-1}\text{ }|\text{ }\textbf v_{p}\in\mathrm{col}\\&\left(\begin{bmatrix}\textbf M_0 & \textbf M_1 &\cdots& \textbf M_{j-1}\end{bmatrix}\right) \text{ for } p=1,2,\dots,k-1\}
\end{align*}
for $j=1,2,\dots,n-1$, has full rank $n$.
\end{theorem}

The notation $\text{col}(\cdot)$ denotes the column space of a matrix. The controllability of the reduced HPDS with linear input \eqref{eq:redsysco} can be stated as follows.

\begin{proposition}\label{pro:control}
    If the tensor-based HPDS of with linear input \eqref{eq:sysco} for even $k$ is strongly controllable, then the reduced HPDS with linear input \eqref{eq:redsysco} is also strongly controllable.
\end{proposition}
\begin{proof}
    To prove that reduced HPDS with linear input (\ref{eq:redsysco}) for even $k$ is strongly controllable, it is equivalent to prove that the reduced controllability matrix  
\begin{align*}
\textbf R_{\text{red}}=\begin{bmatrix}\textbf N_0 & \textbf N_1 &\cdots& \textbf N_{r-1}\end{bmatrix},
\end{align*}
where $\textbf N_0=\textbf B_{\text{red}}$ and 
each $\textbf N_j$ is formed from 
\begin{align*}
    \{&\mathscr A_{\text{red}}\times_1\textbf  v_1\times_2\textbf{v}_2\times_3\cdots\times_{k-1}\textbf v_{k-1}\text{ }|\text{ }\textbf v_{p}\in\mathrm{col}\\&\left(\begin{bmatrix}\textbf N_0 & \textbf N_1 &\cdots& \textbf N_{j-1}\end{bmatrix}\right) \text{ for } p=1,2,\dots,k-1\}
\end{align*}
for $j=1,2,\dots,n-1$, has full rank $r$. Without loss of generality, we assume that $m=1$ so that the control matrix reduces to a vector $\textbf{b}$. It is evident that $\textbf N_0=\textbf V^\top\textbf M_0$. With the compact HOSVD format of $\mathscr{A}$, we can compute $\textbf{M}_1$  as
\begin{align*}
\mathscr A\textbf b^{k-1}&=\mathscr S_{\text{red}}\times_1\textbf b^\top\textbf V\times_2\textbf b^\top\textbf V\times_{3}\cdots\times_{k-1}\textbf b^\top\mathbf{V}\times_{k}\mathbf{V}_{k}.
\end{align*}
Moreover, since $\textbf b_{\text{red}}=\textbf V^\top\textbf b$, we can compute $\textbf N_1$ as
\begin{align*}
\mathscr A_{\text{red}}\textbf b_{\text{red}}^{k-1}&=\mathscr S_{\text{red}}\times_1\textbf b^\top\textbf V\times_2\cdots\times_{k-1}\textbf b^\top\mathbf{V}\times_{k}\mathbf{V}^\top\mathbf{V}_{k}\\&=\mathbf{V}^\top(\mathscr S_\text{red}\times_1\textbf b^\top\textbf V\times_2\cdots\times_{k-1}\textbf b^\top\mathbf{V}\times_k \textbf{V}_k).
\end{align*}
Therefore, it holds that $\textbf N_1=\textbf V^\top\textbf M_1$. By extending the same reasoning used to derive $\textbf N_1$, it can be shown that $\textbf N_j=\textbf V^\top\textbf M_j$ for $j=0,1,\dots,n-1$. Hence, the reduced controllability matrix satisfies $\textbf{R}_{\text{red}} = \textbf{V}^\top\textbf{R}$. If the tensor-based HPDS with linear input \eqref{eq:sysco} for even $k$ is strongly controllable, then the controllability matrix $\textbf{R}$ has full rank $n$ according to Theorem \ref{thm:control}. Since the factor matrix \textbf{V} contains orthonormal columns, it follows that the reduced controllability matrix $\textbf{R}_{\text{red}}$ has full rank $r$, and thus the reduced HPDS with linear input \eqref{eq:redsysco} is also strongly controllable.
\end{proof}

This result highlights the effectiveness of the proposed HOSVD-based model reduction in preserving strong controllability for HPDSs with linear input, making it amenable to scalable control design. Additionally, when $k$ is odd, corresponding to HPDSs of even degree, only accessibility, a weaker form of controllability, can be guaranteed for the tensor-based system \eqref{eq:sysco} under the same condition stated in Theorem \ref{thm:control} \cite{chen2021controllability}. By similar arguments, it follows that accessibility is likewise preserved for the reduced HPDS with linear input \eqref{eq:redsysco} under the proposed HOSVD-based model reduction method.

\subsection{Observability}
Since establishing strong observability results for output HPDSs is generally challenging, we focus our analysis on local weak observability defined in nonlinear systems theory. Consider a general tensor-based HPDS with linear output 
\begin{align}\label{eq:sysoo}
		\begin{cases}
			\dot{\textbf x}(t)=\mathscr A\textbf x(t)^{k-1}  \\
			\textbf y(t)=\textbf C\textbf x(t)\end{cases}
\end{align}
with the reduced counterpart obtained from the proposed HOSVD-based model reduction
\begin{align}\label{eq:redsysoo}
\begin{cases}
\dot{\mathbf{z}}(t)=\mathscr{A}_{\text{red}}\mathbf{z}(t)^{k-1}\\
\textbf{y}(t)=\textbf{C}_{\text{red}}\textbf{z}(t)
\end{cases}.
\end{align}
For simplicity, we ignore the control matrix $\textbf{B}$ and its reduced form $\textbf{B}_{\text{red}}$. In addition, we assume that the tensor-based HPDS with linear output \eqref{eq:sysoo} is complete, i.e., for every $\textbf x_0\in M\subset \mathbb{R}^n$, there exists a solution $\textbf x(t)\in  M$ satisfying $\textbf x(0)=\textbf x_0$ and $\textbf x(t)\in M$ for all $t\in\mathbb R$. 

The concept of local weak observability can be characterized through distinguishability \cite{hermann1977nonlinear}. Specifically, given a subset $U\subseteq M\subset \mathbb{R}^n$, two initial states $\textbf{x}_0(0)$ and $\textbf{x}_1(0)$ within $U$ are said to be $U$-distinguishable if their respective trajectories remain in $U$ over a finite time interval and yield different outputs, i.e., $\textbf{C}\textbf{x}_0(t) \neq \textbf{C}\textbf{x}_1(t)$ for some $t$ in this interval. A system is called locally weakly observable at a point $\textbf x_0\in M$ if there exists an open neighborhood $U$ of $\textbf x_0$ such that for every open set $V \subseteq U$ containing $\textbf x_0$, the only point in $V$ that is not $V$-distinguishable from $\textbf x_0$ is $\textbf x_0$ itself. A tensor-based rank condition is also provided to determine the local weak observability of the tensor-based HPDS with linear output \eqref{eq:sysoo} \cite{mao2025tensor}.

\begin{theorem}[Observability \cite{mao2025tensor}]\label{thm:obs}
    The tensor-based HPDS with linear output \eqref{eq:sysoo} is locally weakly observable at $\textbf x$  if and only if the state-dependent observability matrix defined as
\begin{align}
		\textbf O(\textbf x)=\begin{bmatrix}
			\textbf P_0(\textbf{x})\\
			\textbf P_1(\textbf{x})\\ \vdots \\
			\textbf P_{n-1}(\textbf{x})
		\end{bmatrix},
	\end{align}
where $\textbf P_0=\textbf C$ and
\begin{align*}
\textbf P_{1}&=\textbf C\textbf A_{(k)}\sum_{q=1}^{k-1}\textbf x^{[q-1]}\otimes \textbf I_n\otimes\textbf x^{[k-1-q]},\\
\textbf P_{j}&=\textbf C\textbf A_{(k)}\textbf F_2\cdots\textbf F_{j}\sum_{q=1}^{jk-2j+1}\textbf x^{[q-1]}\otimes \textbf I_n\otimes\textbf x^{[jk-2j+1-q]},\\
\textbf F_j&=\sum_{i=1}^{(j-1)k-2j+3}\textbf I_n^{[i-1]}\otimes \textbf A_{(k)}\otimes\textbf I_n^{[(j-1)k-2j+3-i]},
\end{align*} for $j=2,3,\dots,n-1$, has full rank $n$.
\end{theorem}

The superscript ``$[\cdot]$'' denotes the Kronecker product power. We find that if the dynamic tensor $\mathscr{A}$ is symmetric, the local weak observability of the reduced HPDS with linear output \eqref{eq:redsysoo} can be preserved.

\begin{proposition}
    Suppose that the dynamic tensor $\mathscr{A}$ is symmetric. If the tensor-based HPDS with linear output \eqref{eq:sysoo} is locally weakly observable at a point $\textbf{x}$, then the reduced HPDS with linear output \eqref{eq:redsysoo} is also locally weakly observable at the point $\textbf{z} = \textbf{V}^\top \textbf{x}$.
\end{proposition}
\begin{proof}
    To establish that the reduced HPDS with linear output (\ref{eq:redsysoo}) is locally weakly observable at $\textbf z$, it suffices to show that the reduced observability matrix
\begin{align*}
		\textbf O_{\text{red}}(\textbf z)=\begin{bmatrix}
			\textbf Q_0(\textbf{z})\\
			\textbf Q_1(\textbf{z})\\\vdots \\
			\textbf Q_{r-1}(\textbf{z})
		\end{bmatrix}
	\end{align*}
has full rank $r$ at $\textbf{z} = \textbf{V}^\top \textbf{x}$, where $\textbf{Q}_j$ are defined analogously to $\textbf{P}_j$ using $(\textbf{A}_{\text{red}})_{(k)}$ and $\textbf{C}_{\text{red}}$. It is evident that $\textbf{Q}_0 = \textbf{P}_0 \textbf{V}$. Since $\mathscr{A}$ is symmetric, the last factor matrix in its HOSVD $\textbf{V}_k=\textbf{V}$ and $\mathscr{A}_{\text{red}}=\mathscr{S}_{\text{red}}$. With the compact HOSVD of $\mathscr{A}$, we can compute $\textbf{P}_1$ as
\begin{align*} 
			\textbf P_1
            =\textbf C\textbf V(\textbf S_\text{red})_{(k)}\sum_{q=1}^{k-1} (\textbf V^\top \textbf x)^{[q-1]}\otimes \textbf V^\top\otimes(\textbf V^\top\textbf x)^{[k-1-q]}.
            \end{align*}
Moreover, since $\textbf{C}_{\text{red}}=\textbf{C}\textbf{V}$, we can compute $\textbf{Q}_1$ as
\begin{align*}
	&\textbf Q_1=\textbf C_{\text{red}}(\textbf S_\text{red})_{(k)}\sum_{q=1}^{k-1}\textbf z^{[q-1]}\otimes \textbf I_r\otimes \textbf z^{[k-1-q]}\\&=\textbf C\textbf V(\textbf S_\text{red})_{(k)}\Big(\sum_{q=1}^{k-1} (\textbf V^\top x)^{[q-1]}\otimes \textbf V^\top\otimes (\textbf V^\top\textbf x)^{[k-1-q]}\Big)\textbf V.
\end{align*}
Thus, it holds that  $\textbf{Q}_1=\textbf{P}_1\textbf{V}$. By extending the same reasoning, it can be shown that $\textbf{Q}_j=\textbf{P}_j\textbf{V}$ for $j=0,1,\dots,n-1$. Hence, the reduced observability matrix satisfies $\textbf{O}_{\text{red}}(\textbf{z})=\textbf{O}(\textbf{x})\textbf{V}$. If the tensor-based HPDS with linear output \eqref{eq:sysoo} is locally weakly observable at a point $\textbf{x}$, then the observability matrix $\textbf{O}(\textbf{x})$ has full rank $n$ based on Theorem \ref{thm:obs}. Since the factor matrix $\textbf{V}$ contains orthonormal columns, it follows immediately that the reduced observability matrix $\textbf{O}_{\text{red}}(\textbf{z})$ has full rank $r$, and the reduced HPDS with linear output \eqref{eq:redsysoo} is also locally weakly observable at the point $\textbf{z} = \textbf{V}^\top \textbf{x}$.
\end{proof}

We consider tensor-based HPDSs with symmetric dynamic tensors and show that local weak observability is preserved in the reduced HPDS with linear output \eqref{eq:redsysoo} under the proposed HOSVD-based model reduction method. This result ensures that the essential observability characteristics of the high-dimensional system are maintained in the reduced model, supporting reliable state estimation and system monitoring.

\section{Numerical Examples}\label{sec:num}
We demonstrate the effectiveness of the proposed framework through two numerical examples. All simulations were performed in MATLAB R2023b on a machine equipped with an Apple M3 chip and 16 GB of RAM, utilizing Tensor Toolbox 3.6 \cite{osti_1230898}. The code for reproducing these experiments is available at https://github.com/XinMao0/HOSVD-MOR.

\subsection{Stability}
In this example, we evaluated the effectiveness of the proposed HOSVD-based model reduction method in preserving HPDS stability. Specifically, we considered a 6-dimensional HPDS of degree three, which can be equivalently represented by an odeco dynamic tensor $\mathscr{A}\in\mathbb{R}^{6\times 6\times 6\times 6}$. Due to space constraints, the full polynomial form of the system is omitted here and can be found in the accompanying code. The dynamic tensor $\mathscr{A}$ admits a compact HOSVD as
\[
\mathscr{A} = \mathscr{S} \times_1 \textbf V \times_2 \textbf V \times_3 \textbf V\times_4\textbf V,
\] 
where the core tensor $\mathscr S\in  \mathbb{R}^{3 \times 3 \times 3 \times 3}$ is diagonal with nonzero entries $\mathscr S(1,1,1,1)= -8.2880$,
$\mathscr S(2,2,2,2)=-3.2248$, and 
$\mathscr S(3,3,3,3)= -9.7615$, and the factor matrix \textbf{V} is given by
\begin{align*}
\textbf V=\begin{bmatrix}
-0.1743  &  0.0129 &   0.7769\\
   -0.0115  & -0.4458  & 0.2735\\
   -0.0802  &  0.0156  & -0.5407\\
   -0.5370 &  -0.1316 &  -0.1081\\
   -0.4111  &  0.8066   & 0.0856\\
    0.7112  &  0.3646 &   0.1017
\end{bmatrix}.
\end{align*}
By applying the proposed model reduction method, the reduced HPDS is simply computed as
\begin{equation*}
    \begin{cases}
        \dot{z}_1(t) &= -8.2880z^3_1(t)\\
        \dot{z}_2(t) &= -3.2248z^3_2(t)\\
        \dot{z}_3(t) &= -9.7615z^3_3(t)\\
    \end{cases}.
\end{equation*}
This reduction significantly decreases the total number of system parameters from 1296 to 81 (or just 3 when considering only the nonzero entries). Additionally, since the dynamic tensor $\mathscr{A}$ contains non-positive Z-eigenvalues, the original HPDS is stable. Therefore, by Proposition \ref{pro:stability}, the reduced model is asymptotically stable at the origin. For example, consider the initial condition $\textbf x_0=[0.3341\text{ }    2.8115 \text{ }  -1.2861  \text{ } -1.1378 \text{ }  -1.2017  \text{ } -1.8510]^\top$. As shown in Fig. \ref{fig:2}, simulated with consistent initial conditions (i.e., $\textbf{x}_0$ and $\textbf{z}_0=\textbf{V}^\top\textbf{x}_0$), both original and reduced systems converge to their respective equilibrium points. This demonstrates that the HOSVD-based reduced HPDS effectively preserves the stability properties of the original HPDS.

\begin{figure}[t]
    \centering
    \includegraphics[width=\linewidth]{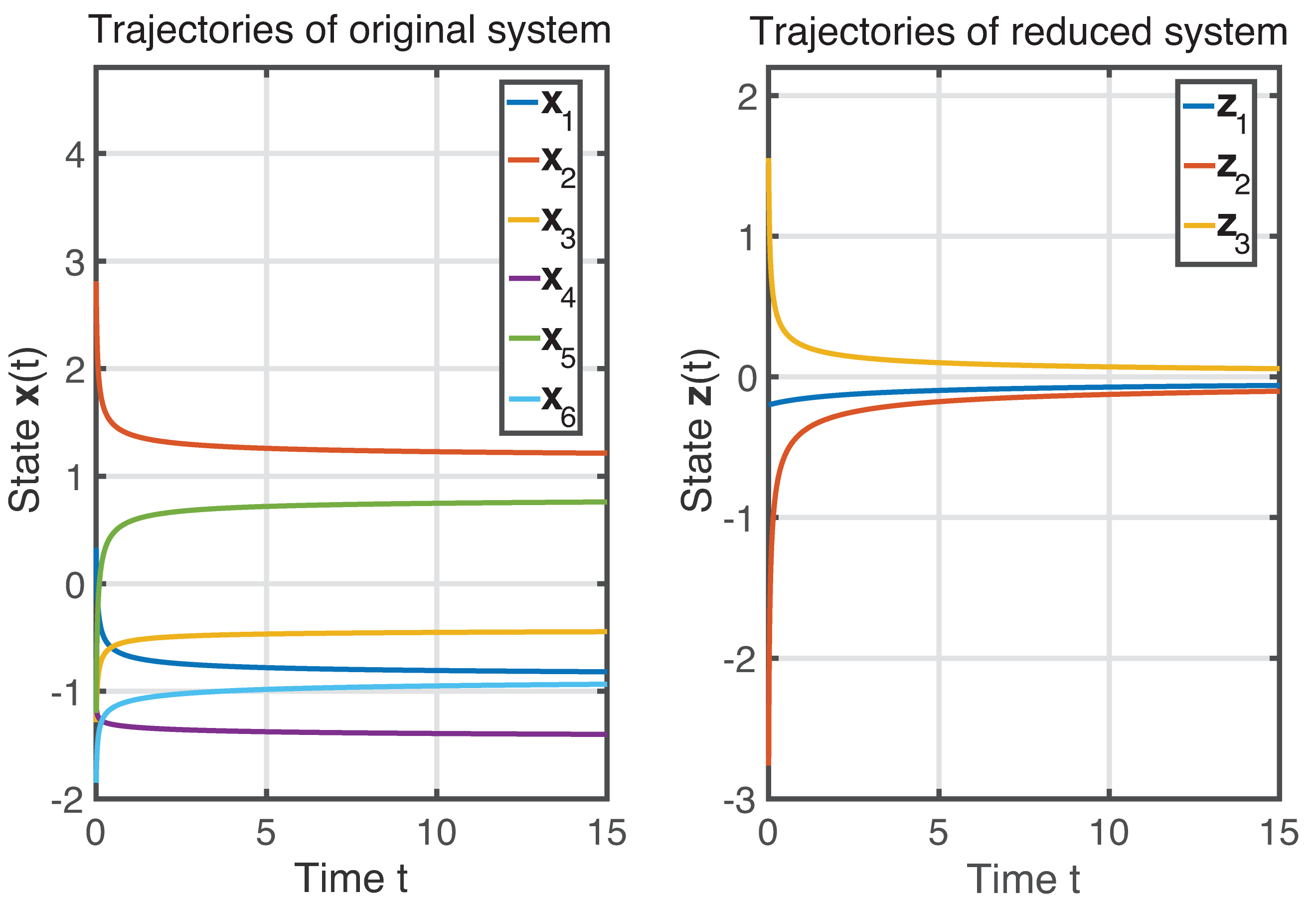}
    \caption{Trajectories of the original and reduced systems with consistent initial conditions.}
    \label{fig:2}
\end{figure}

\subsection{Controllability}
In this example, we examined the preservation of controllability under the proposed HOSVD-based model reduction method. Specifically, we considered a 12-dimensional HPDS of degree three with linear input
\[
\dot{\textbf{x}}(t) = \mathscr{A} \textbf{x}(t)^3 + \textbf{B}\textbf{u}(t),
\]
where $\mathscr{A} \in \mathbb{R}^{12 \times 12 \times 12\times 12}$ is an almost symmetric fourth-order tensor and $\textbf{B} \in \mathbb{R}^{12 \times 5}$ is a control matrix. The entries of both $\mathscr{A}$ and $ \textbf{B}$ are independently drawn from the standard normal distribution. To verify its controllability, we constructed the controllability matrix as defined in \eqref{eq:controllability} and computed its rank, which is equal to 12. Thus, the HPDS with linear input is strongly controllable. 
We then applied the proposed HOSVD-based model reduction technique to obtain a reduced model characterized by
$\mathscr{A}_{\text{red}}\in\mathbb{R}^{7\times 7\times 7\times 7}$ and $\textbf{B}_{\text{red}}\in\mathbb{R}^{7\times 5}$. This reduction reduces the total number of system parameters from 20796 to 2436, offering substantial memory and computational savings. Furthermore, we computed the controllability matrix with $\mathscr{A}_{\text{red}}$ and $\textbf{B}_{\text{red}}$ and observed that its rank is equal to 7, which implies that the reduced HPDS with linear input is also strongly controllable. Consistent with Proposition \ref{pro:control}, this result confirms the effectiveness of the HOSVD-based model reduction method in preserving the controllability of HPDSs with linear input.

\section{Conclusion}\label{sec:con}
In this article, we introduced an innovative model reduction method for homogeneous polynomial dynamical systems (HPDSs) by leveraging tools from tensor theory. Recognizing that HPDSs can be naturally represented by higher-order tensors, we exploited the compact structure revealed by higher-order singular value decomposition (HOSVD) to systematically reduce model complexity while retaining the essential characteristics of the original system. The reduced models constructed via dominant mode projection not only preserve the homogeneous polynomial structure but also maintain key system-theoretic properties, including stability, controllability, and observability. These theoretical guarantees were rigorously established through conditions linking the dynamics of the original and reduced systems.  Numerical examples demonstrate the effectiveness of the proposed reduction method in maintaining essential system behavior while significantly reducing state dimensionality. Furthermore, by employing a homogenization procedure, our method can be extended to general polynomial dynamical systems, thereby significantly broadening its utility across nonlinear modeling domains.

Future work will focus on applying the proposed HOSVD-based model reduction method to various real-world systems characterized by higher-order interactions. In particular, we aim to explore its applicability to ecological networks, where such interactions play a critical role in shaping community structure and function. By reducing the complexity of these systems while preserving essential dynamical properties, our method can improve the predictive modeling of coarse-grained dynamics and offer new insights into the identification and classification of ecological enterotypes within complex communities. Additionally, it is worthwhile to extend the proposed method by establishing connections to controllability and observability Gramians, paving the way for a tensor-based analogue of balanced truncation. Finally, integrating our method with data-driven techniques, such as the eigensystem realization algorithm \cite{juang1985eigensystem}, could enable learning reduced models directly from high-dimensional time-series data, extending the applicability of our method to systems where governing equations are partially known or entirely unavailable.

\section*{References}
\bibliographystyle{IEEEtran}
\bibliography{reference}

\end{document}